\documentclass[11pt]{amsart}

\usepackage{AFBoix2015}

\begin{document}

\author{Alberto F.\,Boix}
\address{Department of Mathematics, Universitat Polit\`ecnica de Catalunya BarcelonaTech, Av. Eduard Maristany 16, 08019, Barcelona, Spain.}
\email{alberto.fernandez.boix@upc.edu}

\author{Majid Eghbali}
\address{School of Mathematics, Institute for Research in Fundamental Sciences (IPM), P. O. Box 19395-5746, Tehran, Iran}
\email{majideghbali83@gmail.com}
\email{m.eghbali@yahoo.com}

\title[On Hellus--Lyubeznik--Yildirim's conjecture]{On Hellus--Lyubeznik--Yildirim's conjecture of local cohomology modules}

\keywords{Local cohomology, cosupport}

\subjclass[2020]{Primary 13D45, 14B15 13H05}

\begin{abstract}
The goal of this paper is to study the so--called Hellus--Lyubeznik--Yildirim (HLY) conjecture, that predicts the following: given a regular local ring $(R,\fm)$, and any ideal $I\subset R$, zero is an associated prime ideal of the Matlis dual of any non--zero local cohomology module supported on $I$. Among other results, we give some partial positive answers to this conjecture in the following cases: when $\depth (R/I)=1$, when $\depth (R/I)=2$ under some extra assumptions, when $I$ is a squarefree monomial ideal inside a formal power series ring over a field, and when $R$ is a formal power series over a discrete valuation ring of mixed characteristic.
\end{abstract}

\maketitle

\section*{Introduction}

In what follows, $(R,\mathfrak{m},\mathbb{K})$ will denote a commutative Noetherian local ring with residue field $\mathbb{K}=R/\fm$, $I\subset R$ will be an ideal, and $(-)^{\vee}$ will denote the Matlis duality functor $\Hom_R (-,E)$, where $E$ denotes a choice of injective hull of the residue field of $R$ in the category of $R$--modules. Moreover, given an integer $j\geq 0$, we denote by
\[
H_I^j (M):=\colim_{k\geq 0} \Ext_R^j \left(R/I^k, M\right)
\]
the $j$--th local cohomology group supported on $I$. On the other hand, we also set
\[
\grade (I):=\min\{j\in\mathbb{N}_0:\ H_I^j (R)\neq 0\},\ \cdim (I):=\max\{j\in\mathbb{N}_0:\ H_I^j (R)\neq 0\}.
\]
Finally, we denote by $\sqrt{I}$ the radical of the ideal $I$.


The following conjecture, originally studied by Hellus in \cite{Hellushab}, and formally raised by Lyubeznik and Yildirim in \cite{LyubeznikYildirim} motivates the content of this paper.

\begin{con}[Hellus--Lyubeznik--Yildirim's conjecture]\label{lyubeznik--yildirim conjecture}
Assume, in addition, that $R$ is regular, and let $j\geq 0$. Then, if $H_I^j (R)\neq 0$, then $(0)\in\Ass_R (H_I^j (R)^{\vee})$.
\end{con}
This conjecture stems from the interest, pointed out by Hellus in \cite{Hellushab} to study the associated prime ideals of Matlis duals of local cohomology modules. The interested reader may like to consult \cite[page 2]{Hellushab} for additional information. In what follows, in order to avoid any misunderstanding, we plan to use the following terminology.

\begin{df}\label{precising where the conjecture holds}
Let $(R,\fm)$ be a commutative Noetherian local ring, let $I\subset R$ be an ideal, and let $j\geq 0$ be an integer. We say that \textit{Hellus--Lyubeznik--Yildirim (HLY) conjecture holds at $j$} if $H_I^j (R)\neq 0$ and $(0)\in\Ass_R (H_I^j (R)^{\vee})$. We say that \textit{HLY holds} if it is true for any $\grade (I)\leq j\leq\cdim (I)$.
\end{df}

The first important contributions to this problem were made by Hellus in \cite{Hellushab}. Specifically, the main situation studied in \cite{Hellushab} is when $R=\mathbb{K}[\![x_1,\ldots,x_d]\!]$, where $\mathbb{K}$ is any field, and $I=(x_1,\ldots,x_c)$, for some $1\leq c\leq d$. We want to review this case briefly, borrowing this material from \cite[Introduction]{HelluscoassII}.

\begin{enumerate}[(i)]

\item $\Ass_R (H_{(x)}^1 (R)^{\vee})=\spec (R)-V(x)$, where $V(x)$ denote the set of prime ideals containing $x$. In particular, $(0)\in\Ass_R (H_{(x)}^1 (R)^{\vee})$.

\item $\Ass_R (H_{(x_1,\ldots,x_d)}^d(R)^{\vee})=\{(0)\}$.

\item $\Ass_R (H_{(x_1,\ldots,x_{d-1})}^{d-1}(R)^{\vee})=\{(0)\}\cup\{(f):\ f\text{ prime element,}\ f\notin (x_1,\ldots,x_{d-1})\}$.

\item All the elements of $\Ass_R (H_{(x_1,\ldots,x_{d-2})}^{d-2}(R)^{\vee})$ have height between $0$ and $2$. Moreover,
\[
\begin{cases}
(0)\in\Ass_R (H_{(x_1,\ldots,x_{d-2})}^{d-2}(R)^{\vee}),\\
\{(f):\ \text{ prime element,}\ f\notin (x_1,\ldots,x_{d-1})\}\subseteq\Ass_R (H_{(x_1,\ldots,x_{d-2})}^{d-2}(R)^{\vee}),\\
\fp\in\Ass_R (H_{(x_1,\ldots,x_{d-2})}^{d-2}(R)^{\vee})\text{ and }\height(\fp)=2\Longleftrightarrow\sqrt{\fp+(x_1,\ldots,x_{d-2})}=\sqrt{(x_1,\ldots,x_d)}.
\end{cases}
\]

\end{enumerate}

In general, it is known that HLY conjecture holds when $R$ is regular and of prime characteristic $p$ (see \cite[Theorem 1.1]{LyubeznikYildirim} and \cite[Theorem 1.2]{ZhangFmodules}). This comes from the following more general statement concerning Lyubeznik's $F$--modules \cite[Theorem 1.1]{LyubeznikYildirim}.

\begin{teo}[Lyubeznik--Yildirim]\label{HLY in prime characteristic}
Let $(R,\fm)$ be a complete Noetherian regular local ring of prime characteristic $p$, and let $\mathcal{M}$ be an $F$--finite $F$--module such that $(0)\notin\Ass_R (\mathcal{M})$. Then, we have that $(0)\in\Ass_R (\mathcal{M}^{\vee})$.
\end{teo}

However, the conjecture is quite open when $R$ is regular and of characteristic zero. Actually, keeping in mind that holonomic $\mathcal{D}$--modules play a similar role in characteristic zero for studying finiteness properties of local cohomology modules, and taking into account Theorem \ref{HLY in prime characteristic}, it seems natural for us to ask the following:

\begin{quo}\label{question about coass and holonomicity: intro}
Let $\mathbb{K}$ be a field of characteristic zero, let $R=\mathbb{K}[\![x_1,\ldots,x_d]\!]$, let $\mathcal{D}$ be the ring of $\mathbb{K}$--linear differential operators on $R$, and let $\mathcal{M}$ be a holonomic $\mathcal{D}$--module. If $(0)\notin\Ass_R (\mathcal{M})$, then is it true that $(0)\in\Ass_R (\mathcal{M}^{\vee})$?
\end{quo}
One of the main goals of this paper is to show, on the one hand, that in general Question \ref{question about coass and holonomicity: intro} has a potential negative answer. On the other hand, we also plan to show that, for local cohomology modules supported on squarefree monomial ideals, the answer is positive by carrying out a mild generalization of the arguments used by Lyubeznik and Yildirim for proving the prime characteristic case.

Now, we provide a more detailed overview of the contents of this paper for the benefit of the reader. In Section \ref{section: HLY reformulations and reductions}, we present several equivalent formulations of the HLY conjecture, connecting it specially with coassociated prime ideals and cosupport of modules as studied respectively by Z\"oschinger \cite{Zoschingercoass} and Richardson \cite{Richardsoncosupport}. We also explore the behaviour of HLY under flat base change, showing that, up to the case of the completion map, it is a delicate issue. In Section \ref{section: HLY general results} we present some general positive results. Mainly, we investigate the following:

\begin{quo}\label{Sec2question}
Let $(R,\fm)$ be a regular local ring of dimension $d$ and $I$ be a non-zero ideal with $\depth (R/I)= t$. Is it true that $(0) \in \Ass_R(H^{d-t}_I(R)^{\vee})$?   
\end{quo}
Question \ref{question about coass and holonomicity: intro} in Section \ref{characteristic zero} is the main subject to consider over rings of characteristic zero. Section \ref{Mixed} is devoted to the mixed characteristic case. Our investigation shows that, in that setting, the HLY conjecture is unlikely to be true. This, jointly with the results in Section \ref{section: HLY counterexamples} we show a potential obstruction to this conjecture in the characteristic zero setting. In contrast, in Section \ref{section: HLY and squarefree monomial ideals}, we plan to show that HLY holds for local cohomology supported on squarefree monomial ideals by combining the arguments used by Lyubeznik and Yildirim jointly with the techniques developed by Singh and Walther \cite{SinghWalther2007} to study local cohomology modules under the action of pure morphisms. This technique has been already shown as fruitful to prove interesting statements about local cohomology modules (see for instance \cite{AlvarezMontaner2015,BoixEghbali2023IJAC,BoixEghbaliRendiconti}).

\section{Reformulations of HLY and possible reductions}\label{section: HLY reformulations and reductions}

The main goal of this section is to present several equivalent formulations of the HLY conjecture that connects it with some classical notions explored in the framework of Commutative Algebra. More precisely, our first aim is to link HLY conjecture with the notion of coassociated prime ideals. The set of coassociated prime ideals of a module $M$ over a commutative Noetherian local ring $(S,\fm)$ consists of the $\fp\in\spec (S)$ such that there is an $S$--module surjection $\xymatrix@1{M\ar@{->>}[r]& L}$ with $L\subseteq E$ such that $\fp=(0:_S L)$, see \cite{Zoschingercoass} and \cite{Yassemicoassociated}. We also define $\cosupp_S (M)$, the cossuport of $M$ (in the sense of Richardson \cite{Richardsoncosupport}) as the support of the Matlis dual of $M$.

Our next statement present several different equivalent reformulations of the HLY conjecture.

\begin{lm}[Reformulations of HLY conjecture]\label{reformulations of HLY}
Let $(R,\fm)$ be a commutative Noetherian regular local ring of dimension $d$, let $I\subset R$ be an ideal, and let $j\geq 0$. Then, the following statements are equivalent.

\begin{enumerate}[(i)]

\item If $H_I^j (R)\neq 0$, then $(0)\in\Ass_R (H_I^j (R)^{\vee})$.

\item If $H_I^j (R)\neq 0$, then $\mu_0 ((0),H_I^j (R)^{\vee})>0$, where $\mu_0 ((0),(-))$ denotes the $0$--th Bass number with respect to the zero ideal.

\item If $H_I^j (R)\neq 0$, then $(0)\in\coass_R(H_I^j (R))$. That is, there is a surjective $R$--module map $\xymatrix@1{H_I^j (R)\ar@{->>}[r]& L}$, with $L\cong E$.


\item If $H_I^j (R)\neq 0$, then $\cosupp_R (H_I^j (R))=\spec (R)$.

\item If $H_I^j (R)\neq 0$, then $\stsupp_R (H_I^j (R)^{\vee})=\spec (R)$.

\item If $H_I^j (R)\neq 0$, then $(0)\in\Ass_R (H_j^I (E))$, the $j$--th local homology group of $E$.
 We denote by $H_j^I$ the $j$--th right derived functor of the local homology
\[
\Lambda^I (M):=\lim_k \left(\frac{R}{I^k}\otimes_R M\right).
\]

\item If $H_I^j (R)\neq 0$, then 
\[
(0)\in\Ass_R \left(\lim_k H_{\fm}^{d-j} (R/I^k)\right).
\]

\end{enumerate}

\end{lm}

\begin{proof}
The equivalence between (i) and (ii) follows immediately from \cite[3.2.5]{Strooker1990}. The equivalence between (i) and (iii) follows from $\coass_R (H_I^j (R))=\Ass_R (H_I^j (R)^{\vee})$, and this equality was proved by Z{\"o}schinger \cite[Lemma 3.1]{Zoschingerminimax} (see also \cite[Theorem 2.4]{Yassemicoassociated}). The equivalence between (iii) and (iv) goes as follows. The second named author proved \cite[parts (i) and (ii) of Lemma 4.1]{Eghbaliformal} that $\coass_R(H_I^j (R))\subset\cosupp_R (H_I^j (R))$ and that every minimal element of $\cosupp_R (H_I^j (R))$ belongs to $\coass_R (H_I^j (R))$. He also proved \cite[part (1) of Proposition 3.3]{Eghbaliformal} that $\cosupp_R (H_I^j (R))$ is closed under specialization. In this way, it is clear that the equivalence between (iii) and (iv) follows immediately from these two facts. The equivalence between (iv) and (v) follows from $\cosupp_R (H_I^j (R))=\stsupp_R (H_I^j (R)^{\vee})$, which is due to Richardson \cite[part (i) of Theorem 2.7]{Richardsoncosupport}. The equivalence between (i) and (vi) follows from the equality $H_I^j (R)^{\vee}=H_j^I (E)$ proved by Richardson \cite[Proposition 3.1]{Richardsoncosupport}. Finally, the equivalence between (i) and (vii) is a consequence of a classical statement by Ogus \cite[2.2.3]{Ogus1973}. The proof is therefore completed.
\end{proof}

One can be tempted to extend HLY conjecture to any finitely generated $R$--module $M$. However, we can see quickly that this is in general not possible. For that purpose, we plan to use the following local duality statement obtained by Takahashi, Yoshino and Yoshizawa in \cite{TakahashiYoshinoYoshizawa2009}.

\begin{teo}\label{some local duality}
Let $(R,\mathfrak{m})$ be a commutative Noetherian complete regular local ring, let $M$ be a finitely generated $R$--module of dimension $r$, and let $K_M$ be the canonical module of $M$. Then, we have that $H_I^r (M)^{\vee}\cong\Gamma_{\mathfrak{m},I} (K_M)$.
\end{teo}
With Theorem \ref{some local duality} in mind, we plan to show now that HLY conjecture can not be true replacing $R$ for any finitely generated $R$--module $M$ as follows.

\begin{disc}\label{using local cohomology wrt pairs}
Let $(R,\mathfrak{m})$ be a commutative Noetherian complete regular local ring, and let $M$ be a finitely generated $R$--module of dimension $r$. By Theorem \ref{some local duality}, we have, in particular, that
\[
\Ass_R (H_I^r (M)^{\vee})=\Ass_R (\Gamma_{\mathfrak{m},I} (K_M)).
\]
By \cite[Proposition 1.10]{TakahashiYoshinoYoshizawa2009}, we have that $\Ass_R (H_I^r (M)^{\vee})=\Ass_R (K_M)\cap W(\mathfrak{m},I)$, where
\[
W(\mathfrak{m},I)=\{\mathfrak{p}\in\spec (R)\mid\ \sqrt{I+\mathfrak{p}}=\mathfrak{m}\}.
\]
On the other hand, by \cite[part (c) of Lemma 1.9]{Schenzellecturenotes}, we also know that
\[
\Ass_R (K_M)=\{\mathfrak{p}\in\Ass_R (M)\mid\ \dim (R/\mathfrak{p})=r\}.
\]
If, in addition, $R$ is a domain and $M=R$, then $(0)\in\Ass_R (K_R)$. In this case, $(0)\in\Ass_R (H_I^r (M)^{\vee})$ if and only if $I$ is $\mathfrak{m}$--primary.
\end{disc}

\begin{rk}
We can also reach the same conclusion obtained in Discussion \ref{using local cohomology wrt pairs} using \cite[Lemma 4.1]{Schenzel2007}.
\end{rk}

\begin{disc}\label{HLY not in the finitely generated case}
One elementary observation that we can do at this point is that, in order that HLY conjecture holds, our local cohomology module can not be finitely generated.

Indeed, let $(R,\fm)$ be a regular Noetherian local ring, let $I\subset R$ be an ideal, and let $j\geq 0$ be an integer. If $0\neq H_I^j (R)$ is a finitely generated $R$--module, then we have, by \cite[Theorem 2.7 (vi)]{Richardsoncosupport}, that $\cosupp_R (H_I^j (R))=\{\fm\}$, so for that value we can guarantee, thanks to Lemma \ref{reformulations of HLY}, that HLY does not hold for $I$ at the $j$--th spot.

This also shows, in particular, that if HLY holds for $j$, then $j\geq f_I (R)$, where
\[
f_I (R)=\inf\{k\in\mathbb{N}_0:\ H_I^j (R)\text{ is not finitely generated}\}
\]
is Faltings' finiteness dimension relative to $I$ \cite[9.1.3]{BroSha}.

Notice that, in case $R$ contains a field, $H_I^j (R)\neq 0$ if and only if $(0:_R H_I^j (R))=0$ by the Huneke--Koh--Lyubeznik Theorem \cite[Theorem 3.6]{BoixEghbali2018}. In this way, in the equicharacteristic case, $f_I (R)=0$, so the reduction we have explained might only be potentially useful in the mixed characteristic case.
\end{disc}

\subsection{HLY and Matlis--reflexive local cohomology modules}
In this part, we explore the validity of HLY conjecture for the class of Matlis--reflexive modules over a complete regular local ring. First, we review some well--known facts.

\begin{df}
Let $(R,\fm)$ be a local ring, and let $M$ be an $R$--module.

\begin{enumerate}[(i)]

\item The \textit{evaluation map on $M$}, $\mu_M$ is defined as
\begin{align*}
& \xymatrix{M\ar[r]^-{\mu_M}& M^{\vee\vee}}\\
& m\longmapsto \mu_M (m)(f):=f(m),\ f\in M^{\vee}.
\end{align*}
It is known \cite[Chapter 10]{BroSha} that $\mu_M$ is always an $R$--module monomorphism.

\item We say that $M$ is \textit{Matlis--reflexive} provided $\mu_M$ is an $R$--module isomorphism.

\end{enumerate}
\end{df}
Matlis reflexive modules are nicely characterized by Enochs--Zink theorem (\cite[3.4.13]{Enochs1984} and \cite{Zink1974}).

\begin{teo}[Enochs--Zink]
Let $(R,\fm)$ be a local ring. Then, $M$ is Matlis reflexive if and only if it can be embedded into a short exact sequence $\xymatrix@1{0\ar[r]& M'\ar[r]& M\ar[r]& M''\ar[r]& 0}$, where $M'$ is a Noetherian $R$--module, and $M''$ is an Artinian $R$--module.
\end{teo}
With all these ingredients in mind, we can go straight to the following:

\begin{disc}
Let $(R,\fm)$ be a complete regular local ring, and let $M$ be a Matlis reflexive $R$--module. By Enochs--Zink theorem, we have a short exact sequence
\[
\xymatrix{0\ar[r]& M'\ar[r]& M\ar[r]& M''\ar[r]& 0,}
\]
where $M'$ is Noetherian and $M''$ is Artinian. In this way, we have, by \cite[Theorem 2.7 (vi)]{Richardsoncosupport}, that
\[
\cosupp_R (M')=\begin{cases}
\{\fm\},\text{ if }M'\neq 0,\\
\emptyset,\text{ if }M'=0.
\end{cases}
\]
Moreover, we also know that
\[
\cosupp_R (M'')=\bigcup_{\fp\in\coass_R (M'')}V(\fp).
\]
Keeping in mind \cite[Theorem 2.7 (v)]{Richardsoncosupport} that $\cosupp_R (M)=\cosupp_R (M')\cup\cosupp_R (M'')$, and that $\coass_R (M)$ are the minimal elements of $\cosupp_R (M)$, we have that
\[
(0)\in\coass_R (M)\Longleftrightarrow (0)\in\coass_R (M'').
\]
Now, since $M''$ is Artinian, we know the existence of the following short exact sequence
\[
\xymatrix{0\ar[r]& M''\ar[r]& E^{\oplus t}\ar[r]& Q\ar[r]& 0,}
\]
where $t=\dim_{R/\fm}\Soc (M'')$, and $\Soc (M'')$ denotes the socle of $M''$ \cite[1.2.18]{BrunsHerzog1993}. Moreover, we know, by the Hartshorne--Lichtenbaum vanishing theorem \cite[7.3.2]{BroSha}, that $\coass_R (E^{\oplus t})=\coass_R (E)=\{(0)\}$, and notice that the rightmost equality is where we are using that $R$ is regular. In this way, we know that $\spec (R)=\cosupp_R (E)=\cosupp_R (M'')\cup\cosupp_R (Q)$. Summing up, what we can guarantee is that, if $(0)\notin\coass_R (Q)$, then $(0)\in\coass_R (M'')$ and therefore $(0)\in\coass_R (M)$.
\end{disc}

\subsection{HLY conjecture and base change}
The goal of this part is to study HLY conjecture with respect to base change in both directions. We start by proving that HLY ascends pretty well. The precise statement reads as follows.


\begin{lm}[HLY ascend]\label{ascending HLY}
Let $\xymatrix@1{R\ar[r]& S}$ be a ring homomorphism of regular Noetherian local rings with respective maximal ideals $\fm_R$ and $\fm_S$ such that $S$ is a flat $R$--module and that $\dim (R)=\dim (S)$. Let $I\subset R$ be an ideal and let $j\geq 0$ be an integer. Then, if $H_I^j (R)\neq 0$ and $(0)\in\coass_R (H_I^j(R))$, then either $H_{IS}^j (S)=0$ or $(0)\in\coass_S (H_{IS}^j(S))$.
\end{lm}

\begin{proof}
Since $\xymatrix@1{R\ar[r]& S}$ is a flat and local homomorphism of Noetherian local rings, we have \cite[Exercise 5.7.17]{Schoutensultrabook} that $\dim S =
\dim R+ \dim S/\fm_RS$. Since $\dim S = \dim R$, we have $\dim S/\fm_RS = 0$ and consequently $\sqrt{\fm_RS}=\fm_S$.

Assume that $(0)\in\coass_R (H_I^j(R))$ and that $H_{IS}^{j}(S)\neq 0$. Since $(0)\in\coass_R (H_I^{j}(R))$ and $R$ is regular, we have a surjection $\xymatrix@1{H_I^{j}(R)\ar@{->>}[r]& H_{\fm_R}^{\dim (R)}(R)}$. Since tensor product is right exact, tensoring this surjection with $S$ we obtain the surjection $\xymatrix@1{H_I^{j}(R)\otimes_R S\ar@{->>}[r]& H_{\fm_R}^{\dim (R)}(R)\otimes_R S}$. By flat base change for local cohomology modules, this surjection is equivalent to the following one:
\[
\xymatrix{H_{IS}^{j}(S)\ar@{->>}[r]& H_{\fm_R S}^{\dim (R)}(S).}
\]
Since $\dim (R)=\dim (S)$, $\sqrt{\fm_R S}=\fm_S$, and $S$ is regular, this surjection is equivalent to the surjection $\xymatrix@1{H_{IS}^{j}(S)\ar@{->>}[r]& H_{\fm_S}^{\dim (S)}(S)\cong E_S,}$
where $E_S$ denotes a choice of the injective hull of the residue field of $S$ as $S$--module. The proof is therefore completed.
\end{proof}

Our next goal is to discuss whether we can reduce HLY by flat base change. More precisely, we have in mind the following:

\begin{quo}[HLY descend question]\label{HLY and flat base change}
Let $\xymatrix@1{R\ar[r]& S}$ be a flat local ring homomorphism of regular Noetherian local rings. Let $I\subset R$ be an ideal, and let $j\geq 0$ be an integer. Assume that $(0)\in\coass_S (H_{IS}^j (S))$. Is it true that $(0)\in\coass_R (H_I^j (R))$?
\end{quo}
Unfortunately, as pointed out for instance by Yassemi in \cite{Yassemicoasssurvey}, there are no general results connecting coassociated prime ideals under base change. Actually, one of the few known facts is the case when one passes from a local ring to its completion. Indeed, we bring the following result.




\begin{lm}\label{reduce HLY to the complete case}
Let $(R,\fm)$ be a commutative Noetherian regular local ring, let $I\subset R$ be an ideal, and let $j\geq 0$ be an integer. Assume that $(0)\in\coass_{\widehat{R}}(H_I^j (R)\otimes_R \widehat{R})$. Then, we have that $(0)\in\coass_R (H_I^j (R))$.
\end{lm}

\begin{proof}
The result follows immediately combining \cite[Proposition 3.3]{Zoschingercoass} jointly with the well known fact that $\widehat{R}$ is a regular local ring. The proof is therefore completed.
\end{proof}
Inspired by the above results, we can formulate a generalization of Lemma \ref{reduce HLY to the complete case} as follows.

\begin{lm}\label{HLY reduced to the Henselian case: faith}
Let $\xymatrix@1{R\ar[r]& S}$ be a flat local homomorphism of commutative Noetherian regular local rings, let $I\subset R$ be an ideal, and let $j\geq 0$ be an integer. Assume that
\[
\coass_R (H_I^j (R))\supseteq\{\fp\cap R:\ \fp\in\coass_S (H_I^j (R)\otimes_R S)\}.
\]
Then, if $(0)\in\coass_S (H_I^j (R)\otimes_R S)$, then $(0)\in\coass_R (H_I^j (R))$.
\end{lm}
The other known case is when we tensor a module with a finitely generated module. We refer to \cite[Folgerung 3.2]{Zoschingerminimax} for details. Thanks to Z\"oschinger's result, we obtain the following interesting statement.


\begin{lm}[HLY under finite flat extensions]\label{HLY and finite flat extensions}
Let $\xymatrix@1{R\ar[r]& S}$ be a flat local homomorphism of regular local rings such that $S$, regarded as $R$--module by restriction of scalars, is finitely generated. Let $I\subset R$ be an ideal, and let $j\geq 0$ be an integer. Then, we have that
\[
\coass_R (H_I^j (R)\otimes_R S)=\coass_R (H_I^j (R))\cap\stsupp_R (S).
\]
\end{lm}

The next step one can try to do is to consider passing the HLY conjecture from one ideal to another one. Let us formulate things in a more precise way.

\begin{prop}\label{HLY and suitable quotients}
Let $\mathbb{K}$ be a field of characteristic zero, let $R=\mathbb{K}[\![x_1, \ldots, x_d]\!]$ and $I$ be an ideal of $R$. Assume that there exists an ideal $Q\subset R$ such that there is a surjection $\xymatrix@1{H_I^i (R)\ar@{->>}[r]& H_{\fm}^j (R/Q)}$. Then, if $(0)\in\coass_R (H_{\fm}^j (R/Q))$, then we have that $(0)\in\coass_R (H_I^i (R))$.   
\end{prop}

\begin{proof}
The surjection $\xymatrix@1{H_I^i (R)\ar@{->>}[r]& H_{\fm}^j (R/Q)}$ implies that $\coass_R (H_{\fm}^j (R/Q))\subseteq\coass_R (H_I^i (R))$. The proof is therefore completed.
\end{proof}

\begin{rk}\label{when these conditions hold}
The reader might ask when the assumptions of Proposition \ref{HLY and suitable quotients} are satisfied. Indeed, for $i=\dim R$, the conditions hold by \cite{Egh-Sch12}. For $i=\dim R-1$ and in the case that $I= \fa \cap \fb$,  $\sqrt{\fm}=\fa+\fb$, we have the desired epimorphism, as we plan to show at Proposition \ref{disconnected}. In addition, in the case $\stsupp H^i_I(R) \subseteq \{\fm\}$ we have the epimorphism $\xymatrix@1{H_I^i (R)\ar@{->>}[r]& H_{\fm}^{\dim R} (R)}$, see Proposition \ref{zero dimensional case}.
\end{rk}

\section{General results}\label{section: HLY general results}
The aim of this section is to present some general cases where the HLY conjecture holds. To be more precise, we answer Question \ref{Sec2question} in some cases. First, recall that, when $\depth (R/I)= 0$, Grothendieck's vanishing theorem implies that $\cdim (I) \leq \dim (R)$. The first case we have in mind is the one where $\cdim (I)=\dim (R)$. The precise statement can be formulated as follows.

\begin{prop}\label{the top cdim case}
Let $(R,\fm,\mathbb{K})$ be a commutative Noetherian regular local ring of dimension $d$, and let $I\subseteq R$ be an ideal with $\cdim (I)=d$. Then, we have that $\{(0)\}=\coass_R (H_I^d (R))$.
\end{prop}

\begin{proof}
By Lemma \ref{reduce HLY to the complete case} we can assume, without loss of generality, that $R$ is regular and complete with respect to the $\fm$--adic topology. Under this additional assumption, as byproduct of the Non--Vanishing Theorem jointly with the Hartshorne--Lichtenbaum vanishing theorem, we have that $\cdim (I)=d$ if and only if $I$ is $\fm$--primary. In this case, keeping in mind that $R$ is regular and complete, we obtain that $\Ass_R (H_I^d (R)^{\vee})=\Ass_R (H_{\fm}^d (R)^{\vee})=\Ass_R (K_R)=\{(0)\}$, where we are using local duality for local cohomology modules. The proof is therefore completed. 
\end{proof}

The second case we want to prove is the one where our quotient ring $R/I$ has depth $1$. In order to tackle this issue, we plan to recover and extend a result obtained by Hellus and St\"uckrad \cite[Theorem 2.2]{HellusStuckradtop} concerning Matlis dual of top local cohomology modules. The following statement should be compared with \cite[Theorem 3.2.7]{Hellushab}.

\begin{teo}\label{HStextend}
Let $(R,\fm)$ be a regular local ring of dimension $d$ and $I$ be a non-zero ideal with $\depth (R/I)= 1$. Then $(0) \in \coass_R(H^{d-1}_I(R))$.
\end{teo}

\begin{proof} 
Thanks to Lemma \ref{reduce HLY to the complete case} we can assume, without loss of generality, that $R$ is a complete regular local ring. Now, observe that our assumption $\depth (R/I)= 1$ implies that $\cdim (I)\leq d-1$. In this way, by Prime Avoidance \cite[Proposition 1.11]{AtiyahMac} there is a non--zero $y\in R$ not belonging to any minimal prime of $I$ such that $I+(y)$ is $\fm$--primary. These two remarks, jointly with the short exact sequence
\[
\xymatrix{0\ar[r]& H_{(y)}^1 (H_I^{d-1}(R))\ar[r]& H_{\fm}^d (R)\ar[r]& H_{(y)}^0 (H_I^d (R))\ar[r]& 0}
\]
implies that there is an $R$--module isomorphism
\begin{equation} \label{firstiso}
H_{\fm}^d (R)\cong H_{(y)}^1 (H_I^{d-1}(R)).
\end{equation}
On the other hand, we also consider the short exact sequence
\[
\xymatrix{0\ar[r]& R\ar[r]^-{\cdot y}& R\ar[r]& R/(y)\ar[r]& 0.}
\]
Now, by taking the long exact sequence attached to the functor $\Gamma_I$ we obtain the exact sequence $\xymatrix@1{H_I^{d-1}(R)\ar[r]& H_I^{d-1} (R/(y))\ar[r]& H_I^d (R)}$. On the one hand, since $\cdim (I)\leq d-1$ we have, in particular, that $H_I^d (R)=0$. On the other hand, by using the Independence Theorem \cite[4.2.1]{BroSha} jointly with the fact that $I+(y)$ is $\fm$--primary we have that
\[
H_I^{d-1} (R/(y))\cong H_{I+(y)}^{d-1} (R/(y))\cong H_{\fm}^{d-1} (R/(y)).
\]
Moreover, since $y$ is a non--zero divisor on $R$, we have that $R/(y)$ is a Gorenstein local ring of dimension $d-1$, hence $H_{\fm}^{d-1} (R/(y))\cong E_{R/(y)}$, where $E_{R/(y)}=(0:_E y)$; that is, $H_{\fm}^{d-1}(R/(y))$ are the elements of $E$ that are annihilated by $y$. Replacing $y$ by $y^k$ for any $k\geq 1$, we can show that that there is surjection
\begin{equation} \label{firstsur}
\xymatrix{H_I^{d-1}(R)\ar@{->>}[r]^-{\psi}&  H_{(y)}^0(E).}
\end{equation}
Now, consider Brodmann's exact sequence \cite[8.1.2]{BroSha} at the $(n-1)$--spot.
\[
\xymatrix{H_{I+(y)}^{d-1}(R)\ar[r]& H_I^{d-1}(R)\ar[r]& H_I^{d-1}(R)_y\ar[r]& H_{I+(y)}^{d}(R)\ar[r]& H_I^d (R).}
\]
As already observed, we know that $H_{I+(y)}^{d-1}(R)=0=H_I^d (R)$, so this exact sequence boils down to the following short exact one.
\begin{equation}\label{Brodmann exact sequence}
\xymatrix{0\ar[r]& H_I^{d-1}(R)\ar[r]& H_I^{d-1}(R)_y\ar[r]^-{g}& H_{\fm}^{d}(R)\ar[r]& 0.}
\end{equation}
Now, our goal is to construct a surjective map from $H_I^{d-1}(R)$ to $H_{\fm}^d (R)=E$. With that purpose in mind, let $0\neq\zeta\in E$. Since $E=H_{\fm}^d (R)$, and by the aid of (\ref{firstiso}) we have $E=H_{(y)}^1 (H_I^{d-1}(R))$, this implies that any element of $E$ is annihilated by some power of $y$. So, there is $j\geq 1$ such that
$y^j\cdot\zeta=0$. In this case, by \eqref{firstsur} we already know that $\zeta=\psi (\eta_0)$ for some $\eta_0\in H_I^{d-1}(R)$. This finally shows that there is a surjection from $H_I^{d-1}(R)$ to $E$, as we wanted to prove.
\end{proof}

\begin{lm}
Let $(R,\fm)$ be a regular local ring of dimension $d$, and let $I,J\subset R$ be two ideals. Suppose that there exists a non-zero ring homomorphism $\xymatrix@1{R/I\ar[r]^{\psi}& R/J}$ such that it induces the surjection $\xymatrix@1{H_{\fm}^j (R/I)\ar@{->>}[r]& H_{\fm}^j (R/J)}$ for every $j$. If $\depth R/J=1$, then $(0)\in \coass_R(H^{d-1}_I(R))$. 
\end{lm}

\begin{proof}
It is clear by the assumptions that $\depth R/I \leq 1$. By virtue of Theorem \ref{HStextend}, it is enough to show that $\depth R/I=1$. In the contrary, suppose that $\depth R/I=0$, that is $\fm \in \Ass R/I$. It forces to have a nonzero homomorphism $\xymatrix@1{R/\fm\ar[r]^{\phi}& R/I}$. Now combining $\psi$ and $\phi$ to produce a non zero homomorphism from $R/\fm$ to $R/J$, yields that $\depth R/J=0$, a contradiction. 
\end{proof}

For a local ring $(R,\fm)$, Hartshorne's result \cite[Proposition 2.1]{Hartshorne1962} says that if $\spec R \setminus \{\fm\}$ is disconnected, then the depth of $R$ is at most $1$. This connectedness statement was generalized by Hellus in \cite{Hellusconnectedness} for cohomologically complete intersection ideals, see also \cite[Theorem 3.2 and Corollary 3.3]{Schenzelconnectedness}.

\begin{prop} \label{disconnected}
Let $(R,\fm)$ be a regular local ring of dimension $d$ and $I \neq 0$ be an ideal of $R$ such that $(0)\notin\Ass_R (R/I)$. Suppose that $V(I) \setminus \{\fm\}$ is disconnected. Then, we have that $(0) \in \coass_R (H^{d-1}_I(R))$.   
\end{prop}

\begin{proof}
As $V(I) \setminus \{\fm\}$ is disconnected, there are two ideals $\fa,\ \fb$ of $R$ satisfying the following properties:
\[
\sqrt{I}=\sqrt{\fa \cap \fb},\ \sqrt{\fa+\fb}=\fm,\ \sqrt{\fa}\neq \fm,\ \sqrt{\fb} \neq \fm.
\]
Then the Mayer–Vietoris sequence (which in this case is also known as Rung's display) implies the epimorphism $\xymatrix@1{H^{d-1}_I(R)\ar@{->>}[r]&  H^{d}_{\fm}(R)}$, which is what we need to show.
\end{proof}
In the case that $R$ is a $d$-dimensional complete regular local ring that contains its residue field, which is separably closed, the so-called second vanishing Theorem asserts that whenever $\dim(R/I) \geq 2$, the vanishing of $H^{d-1}_I(R)$ is equivalent to the connectedness of the punctured spectrum of $R/I$ (cf. \cite{HunLyu}, \cite{Ogus1973}, \cite{PeskineSzpiro1973}).
In order to gather the related results it could be interesting to recall the following statement obtained by Boix and Eghbali in \cite[Theorem 1.1(2)]{BoixEghbali2023IJAC}.

\begin{teo}
Let $R$ be a regular local ring of dimension $d$ containing a field and $I$ be of pure dimension $2$. Then the HLY conjecture holds for $H^{d-1}_I(R)$.   
\end{teo}
Our next purpose is to generalize the argument presented along the proof of Proposition \ref{disconnected}. The key tool is the next lemma, which may be regarded as a mild generalization of the argument used for instance in \cite[Theorem 2.1]{Eghbalilcci}. We want to single out that the argument presented here is often useful to prove vanishing statements of local cohomology modules, see for instance \cite[Proof of Lemma 3.2]{Bataviacdim}.

\begin{lm}\label{general Rung display}
Let $(R,\fm)$ be a regular local ring, let $I\subset R$ be an ideal, and set $c:=\cdim (I)$. Assume that there are ideals $J_1,\ J_2$ such that $\sqrt{I}=J_1\cap J_2$, $\max\{\cdim (J_1),\cdim (J_2)\}\leq c$, and $\cdim (J_1+J_2)=c+1$. If $(0)\in\operatorname{Coass}(H_{J_1+J_2}^{c+1}(R))$, then we have that $(0)\in\operatorname{Coass}(H_{I}^{c}(R))$.
\end{lm}

\begin{proof}
We consider the Mayer--Vietoris exact sequence attached to the decomposition $\sqrt{I}=J_1\cap J_2$ at the spot $j=\cdim (I)=c$. Then, we end up with the following exact sequence:
\[
\xymatrix{H_I^c (R)\ar[r]& H_{J_1+J_2}^{c+1}(R)\ar[r]& H_{J_1}^{c+1} (R)\oplus H_{J_2}^{c+1}(R)\ar[r]& H_I^{c+1}(R).}
\]
Since $c=\cdim (I)$, it is clear that $H_I^{c+1}(R)=0$. Moreover, since $\max\{\cdim (J_1),\cdim (J_2)\}\leq c$, we also know that $H_{J_1}^{c+1} (R)\oplus H_{J_2}^{c+1}(R)=0$. Summing up, the above piece of the Mayer--Vietoris long exact sequence gives rise to the surjection $\xymatrix@1{H_I^c (R)\ar@{->>}[r]& H_{J_1+J_2}^{c+1}(R)}$. Finally, since by assumption there is a surjection $\xymatrix@1{H_{J_1+J_2}^{c+1}(R)\ar@{->>}[r]& E}$, we can compose both surjections to obtain the surjection $\xymatrix@1{H_{I}^{c}(R)\ar@{->>}[r]& E}$, which is what we want to prove.
\end{proof}

With exactly the same proof, we can present this slightly more general statement.

\begin{lm}\label{general Mayer Vietoris trick}
Let $(R,\fm)$ be a regular local ring, let $j\geq 0$ be an integer, and let $I\subset R$ be an ideal. Assume that there are ideals $J_1,\ J_2$ such that $\sqrt{I}=J_1\cap J_2$, $H_{J_1}^{j+1}(R)=0=H_{J_2}^{j+1}(R)$, and $(0)\in\coass_R (H_{J_1+J_2}^{j+1}(R))$. Then, we have that $(0)\in\coass_R (H_I^j (R))$.
    
\end{lm}

\begin{rk}
Let $(R,\fm)$ be a regular local ring of dimension $d$, let $I\subset R$ be an ideal. Assume that there are ideals $J_1,\ J_2$ such that $\sqrt{I}=J_1\cap J_2$ and $\cdim (J_1)=\cdim (J_2)=c$. We consider the following cases.

\begin{enumerate}[(i)]

\item  If $\sqrt{J_1+J_2}=\fm$ and $c=d-1$, then Lemma \ref{general Rung display} implies that $(0)\in\coass_R (H_I^{d-1}(R))$. 

\item If $\sqrt{J_1+J_2}=J$, where $J$ is a $1$-dimensional ideal and $0$ is not in $\Ass R/J$, by virtue of Theorem \ref{HStextend}, $(0) \in \coass_R(H^{d-2}_I(R))$. 
\end{enumerate}

\end{rk}
In order to find another potential positive answer to Question \ref{Sec2question} we consider the best possible upper bound on the cohomological dimension of an ideal, obtained by Lyubeznik in \cite[Theorem 1.1]{Lyu2006} (see also \cite[Theorem 3.8]{DaoTakagicdim}, and \cite{Bataviacdim} and the references given therein for more bounds in this spirit). Before reviewing it, we recall the following construction.

\begin{df}
Let $(R,\fm)$ be a local ring, let $I\subset R$ be an ideal, and let $\sqrt{I}=\fp_1\cap\ldots\cap\fp_n$ be the prime decomposition of $\sqrt{I}$. Set $\Delta$ as the simplicial complex constructed as follows: a simplex $\{i_0,\ldots, i_s\}$ belongs to $\Delta$ if and only if $\sqrt{\fp_{i_0}+\ldots+\fp_{i_s}}\neq\fm$.
\end{df}
With this construction in mind, we are ready for reviewing Lyubeznik's upper bound on the cohomological dimension of an ideal.

\begin{teo}[Lyubeznik]\label{Lyubeznik upper bound}
Let $R=\mathbb{K}[\![x_1,\ldots,x_d]\!]$, where $\mathbb{K}$ is a separably closed field, let $I\subset R$ be an ideal, and let $\sqrt{I}=\fp_1\cap\ldots\cap\fp_n$ be the prime decomposition of $\sqrt{I}$. Assume that the height of all the $\fp_i$'s is at most $c$. Set
\[
t:=\left\lceil\frac{d-2}{c}\right\rceil,\ v:=d-t-1.
\]
Then, we have an isomorphism of $R$--modules $H_I^{v+1} (R)\cong E^{\oplus w}$, where $w:=\dim_{\mathbb{K}}\widetilde{H}_{t-1}(\Delta;\mathbb{K})$ and $\widetilde{H}$ denote reduced simplicial homology.
\end{teo}
As immediate consequence of Lyubeznik's upper bound, we obtain the following:

\begin{rk}
Let $R=\mathbb{K}[\![x_1,\ldots,x_d]\!]$, where $\mathbb{K}$ is a separably closed field, let $I\subset R$ be an ideal, and let $\sqrt{I}=\fp_1\cap\ldots\cap\fp_n$ be the prime decomposition of $\sqrt{I}$, where $\depth R/I \geq 2$.

Fix $1\leq i\leq n$. By \cite[2.1.4]{BrunsHerzog1993}, we know that $\height(\fp_i)+\dim (R/\fp_i)=d$. We claim that $\height \fp_i \leq d-2$, which keeping in mind this equality is equivalent to say that $\dim R/\fp_i \geq 2$. Suppose that $\dim R/\fp_i=0$. This implies that $\fm \in \Ass (R/I)$, which is impossible. As $\depth R/I \geq 2$, by Hartshorne's connectedness result \cite[Proposition 2.1]{Hartshorne1962}, the punctured spectrum of $R/I$ is disconnected. In this way, if $\dim R/\fp_j=1$ for some $j$, then $\fp_j + \bigcap_{j \neq i}\fp_i$ is $\fm$-primary, and therefore it yields that $\spec R/I \setminus \{\fm\}$ is disconnected, a contradiction. Summing up, as the punctured spectrum of $R/I$ is connected, $\Delta$ is connected. This implies that $\widetilde{H}_{0}(\Delta;\mathbb{K})=0$, hence we can guarantee that $\cdim (I) \leq d-2$.

Now, suppose that $\height \fp_i \leq \frac{d-2}{2}$ for all $1 \leq i \leq n$. Then, using Theorem \ref{Lyubeznik upper bound} we can ensure that  $(0)\in\operatorname{Coass}(H_I^{d-2} (R))$ if and only if $\dim_{\mathbb{K}}\widetilde{H}_{1}(\Delta;\mathbb{K})\geq 1$.
\end{rk}

We end up this section by recalling a couple of very important results obtained by Lyubeznik in \cite[Theorem 2.4 (a)]{Lyubeznik1993Dmod} and \cite[Theorem 1.4]{Lyu1997}, because there will be very useful for us in subsequent sections.

\begin{prop}\label{zero dimensional case}
Let $(R,\fm, \mathbb{K})$ be a regular local ring containing a field, which we denote by $k$. Let $M$ be an $R$--module such that $\stsupp (M)\subseteq\{\fm\}.$ Moreover, assume that one of the following hypothesis is satisfied

\begin{enumerate}[(i)]

\item  $k$ has characteristic zero, and $M$ is a left $\mathcal{D}$--module.

\item $k$ has prime characteristic and $M$ is an $F$--module.

\end{enumerate}
Then, we have that $M$ is isomorphic, as $R$--module, to a finite direct sum of copies of the injective hull of $\mathbb{K}$.
\end{prop}

\section{The characteristic zero case} \label{characteristic zero}
The goal of this section is to explore the HLY conjecture when $(R,\fm,\mathbb{K})$ is a regular local ring containing a field $k$ of characteristic zero. Unless otherwise is specified, in what follows $\mathcal{D}$ will denote the ring of $k$--linear differential operators on $R$. 

Notice that, thanks to Lemma \ref{reformulations of HLY}, HLY conjecture is connected with the following particular case of a result by Hartshorne and Polini \cite[Theorem 5.1]{HartshornePolinisimple}. In what follows, by $H_{dR}^*$ we mean algebraic de Rham cohomology as worked out by Hartshorne in \cite{Har1975}.

\begin{teo}\label{Hartshorne Polini theorem}
Assume that $R=\mathbb{K}[\![x_1,\ldots,x_n]\!]$, where $\mathbb{K}$ is a field of characteristic zero, let $I\subset R$ be an ideal, let $j\geq 1$, and set $m:=\dim_{\mathbb{K}}H_{dR}^n (H_I^j (R))$. Then, there is a surjective homomorphism of $\mathcal{D}$--modules $\xymatrix@1{H_I^j (R)\ar@{->>}[r]& E^m}$.
\end{teo}
As immediate consequence of Theorem \ref{Hartshorne Polini theorem} one obtains another equivalent formulation of the HLY conjecture over any complete regular local ring containing a field of characteristic zero.

\begin{cor}\label{derhamdim}
Assume that $R=\mathbb{K}[\![x_1,\ldots,x_n]\!]$, where $\mathbb{K}$ is a field of characteristic zero, let $I\subset R$ be an ideal, let $j\geq 0$, and assume that $H_I^j (R)\neq 0$. Then, HLY conjecture holds at $j$ if and only if $\dim_{\mathbb{K}}H_{dR}^n (H_I^j (R))\geq 1$.
\end{cor}

Actually, in characteristic zero, the following reduction is possible. Before presenting it, we review the following notion introduced in \cite{SarriaCallejasCaro2017}.

\begin{df}\label{diff add algebras}
Let $k$ be a field of characteristic zero, and let $(R,\fm,\mathbb{K})$ be a commutative Noetherian regular local ring of Krull dimension $d$ containing $k$. We say
that $R$ is \textit{differentiable admissible} if the following two assertions hold.

\begin{enumerate}[(i)]

\item $\mathbb{K}$ is an algebraic extension of $k$. In the terminology coined by Matsumura \cite{Matsumuraqcf}, $k$ is a \textit{quasi--coefficient field} for $R$.

\item $\operatorname{Der}_k (R)$ is a finitely generated, free $R$--module of rank $d$ and the canonical map
\[
\xymatrix{R_{\fm}\otimes_R \operatorname{Der}_k (R)\ar[r]& \operatorname{Der}_k (R_{\fm})}
\]
is an isomorphism.

\end{enumerate}
\end{df}

In fact, one of the goals of this paper is to study Question \ref{question about coass and holonomicity: intro}, that asks whether HLY for any holonomic $\mathcal{D}$--module. Our next aim now is to give some equivalent ways to formulate Question \ref{question about coass and holonomicity: intro} that are specific for the characteristic zero case. 

\begin{lm}\label{reduction for diff. adm. alg.}
Let $k$ be a field of characteristic zero, and let $(R,\fm,\mathbb{K})$ be a local differentiable admissible $k$--algebra of dimension $d$. Let $I\subset R$ be an ideal and let $j\geq 0$ be an integer. Then, the following are equivalent.

\begin{enumerate}[(a)]

\item If $H_I^j (R)\neq 0$, then there is an $R$--module epimorphism $\xymatrix@1{H_I^j (R)\ar@{->>}[r]& H_{\fm}^d (R)}$.

\item If $H_I^j (R)\neq 0$, then there is a $\mathcal{D}$--module epimorphism $\xymatrix@1{H_I^j (R)\ar@{->>}[r]& H_{\fm}^d (R)}$.

\item If $H_I^j (R)\neq 0$, then there is a $\widehat{\mathcal{D}}$--module epimorphism $\xymatrix@1{H_{I\widehat{R}}^j (\widehat{R})\ar@{->>}[r]& H_{\fm\widehat{R}}^d (\widehat{R})}$, where $\widehat{\mathcal{D}}$ denotes the ring of $\mathbb{K}$--linear differential operators on $\widehat{R}\cong\mathbb{K}[\![x_1,\ldots,x_d]\!]$.

\end{enumerate}
    
\end{lm}

\begin{proof}
Under our assumptions, both $H_I^j (R)$ and $H_{\fm}^d (R)$ have a natural structure of holonomic $\mathcal{D}$--modules. This shows that (a) and (b) are equivalent. On the other hand, (b) implies (c) by combining \cite[Theorem 2.2 and Proposition 2.5]{SarriaCallejasCaro2017}. Finally, part (c) implies (a) by using Lemma \ref{reduce HLY to the complete case}. The proof is therefore completed.
\end{proof}
The following question asks, loosely speaking, whether it is possible to prove HLY using some kind of resolution of singularities. More precisely:

\begin{quo} \label{quoembedd}
Let $(R,\fm)$ be a regular local ring, let $I\subset R$ be an ideal such that $H_I^j (R)\neq 0$ and $(0)\in\coass_R (H_I^j (R))$ for some integer $j\geq 0$. Is there any injective ring homomorphism $\xymatrix@1{R\ar[r]& S}$, where $S$ is a regular local ring and an ideal $J\subset S$ such that $IS\subset J$ and $(0)\in\coass_S (H_J^j (S))$?
\end{quo}
In order to give a partial positive answer to this question, we plan to use the following non--trivial fact proved by Switala \cite[Lemma 2.18]{Switala2017}.

\begin{lm}\label{Switala ascend}
Let $\mathbb{K}$ be a field of characteristic zero, let $R=\mathbb{K}[\![x_1, \ldots, x_n]\!]$ and let $I$ be an ideal of $R$. Set $R':=R[\![z]\!]$ and $I':=IR'+(z)$. Then, for all $q$ and $j$ there is a $\mathbb{K}$--vector space isomorphism
\[
H_{dR}^q (H_I^j (R))\cong H_{dR}^{q+1} (H_I^{j+1}(R)).
\]
\end{lm}
Now, we are in position to give a partial positive answer to Question \ref{quoembedd} as follows.

\begin{prop} \label{HLYembedd}
Let $\mathbb{K}$ be a field of characteristic zero, let $R=\mathbb{K}[\![x_1, \ldots, x_n]\!]$, and let $I$ be an ideal of $R$. Suppose that $S=\mathbb{K}[\![x_1, \ldots, x_n, x_{n+1}, \ldots, x_m]\!]$, $m >n$ and $I'=IS+(x_{n+1}, \ldots, x_m)$ is an ideal of $S$. Then, the HLY conjecture holds for $R, I$ with respect to the integer $j$ if and only if the same holds for $S, I'$ with respect to the integer $j+m-n$. 
\end{prop}

\begin{proof}
By assumption, there is a surjection $\xymatrix@1{H_{I}^j(R)\ar@{->>}[r]&  H_{\fm R}^n (R)}$. Applying $H^n_{dR}(-)$ to the previous surjection, we obtain the following surjection $\xymatrix@1{H^n_{dR}(H_{I}^i(R))\ar@{->>}[r]&  H^n_{dR}(H_{\fm R}^n (R))\cong \mathbb{K}}$ of $\mathbb{K}$--vector spaces. By applying Lemma \ref{Switala ascend} $m-n$ times, this surjection can be identified with the surjection $\xymatrix@1{H^{m=n+m-n}_{dR}(H_{I'}^{j+m-n}(R))\ar@1{->>}[r]& \mathbb{K}}$.
Now we are done thanks to Corollary \ref{derhamdim}.
\end{proof}


\subsection{The essentially finite type case}

Our goal now is proving the HLY conjecture for some regular local rings that are essentially of finite type over a field of characteristic zero, following mainly the approach developed by Dao and Takagi along \cite[Theorem 2.4]{DaoTakagicdim}. The first main result of this part reads as follows.

\begin{prop} \label{generalize for depth}
Let $R=\mathbb{C}[x_1, \ldots, x_n]_{(x_1, \ldots, x_n)}$ and let $I\subset R$ be an ideal with $\depth (R/I)=2$ such that $H^2_{\fm}(R/I)$ is a finite dimensional $R/\fm$--vector space, where $\fm=(x_1, \ldots, x_n) R$. If $H_{I}^{n-2}(R) \neq 0$, then $(0)\in\coass_R (H_I^{n-2}(R))$.
\end{prop}

\begin{proof}
We will show that $H_I^{n-2}(R)$ is supported at the maximal ideal. In that case, the result will follow by Proposition \ref{zero dimensional case}, because in that case $H_I^{n-2}(R)$ is isomorphic to a finite number of copies of $E\cong H_{\fm_R}^n (R)$. Hereafter, we plan to use the same strategy employed by Dao and Takagi along the proof of \cite[Theorem 2.4]{DaoTakagicdim}.

Indeed, we must show that, for any $I \subseteq \fp \in \spec R$ with  $\fp \neq \fm$, $H_{IR_{\fp}}^{n-2}(R_{\fp})=0$. As $\dim R_{\fp}=\height (\fp)$ and $H_{IR_{\fp}}^{n-2}(R_{\fp})=0$ for $\height (\fp) < n-2$ So, it remains to check the vanishing when $\height (\fp)\in\{n-2, n-1\}$. By the Lichtenbaum-Hartshorne vanishing Theorem, $H_{IR_{\fp}}^{n-2}(R_{\fp})=0$, whenever $\height (\fp)=n-2$. Therefore, it is enough to show the vanishing for $\height (\fp)=n-1$. Using \cite[Theorem 2.11]{Ogus1973}, we need to prove $\depth R_{\fp}/I R_{\fp} \geq 2$. By the assumption, $H^2_{\fm}(R/I)$ is a finite dimensional $R/\fm$--vector space. It is the same to suppose that its dual $\Ext^{n-2}_R(R/I,R)$ is supported at the maximal ideal. On the other hand, since $\depth R/I \geq 2$, using local duality once again we obtain $\Ext^{n-1}_R(R/I,R)=0$. Summing up, we may have $\Ext^{i}_{R_{\fp}}(R_{\fp}/I R_{\fp},R_{\fp})=0$ for $i\in\{n-1, n-2\}$. This finishes the proof.
\end{proof}
We can extend a bit more Proposition \ref{generalize for depth} as follows.

\begin{prop}\label{the real e.f.t case}
Let $R=\mathbb{R}[X_1,\ldots,X_d]_{(X_1,\ldots, X_d)}$ and let $I\subset R$ be an ideal with $\depth (R/I)=2$ such that $H^2_{\fm}(R/I)$ is a finite dimensional $R/\fm$--vector space, where $\fm=(X_1, \ldots, X_d) R$. If $H_{I}^{d-2}(R) \neq 0$, then $(0)\in\coass_R (H_I^{d-2}(R))$.
\end{prop}

\begin{proof}
The extension $\xymatrix@1{R\ar@{^(->}[r]& \mathbb{C}[X_1,\ldots,X_d]_{(X_1,\ldots, X_d)}}$ is a flat, local and finite ring homomorphism. Moreover, $\depth (R/I\otimes\mathbb{C})=\depth (R/I)=2$, and $H_{\fm}^2 (R/I)\otimes\mathbb{C}$ is a finite dimensional $\mathbb{C}$--vector space. Therefore, thanks to Lemma \ref{HLY and finite flat extensions}, we can reduce to the case where $R=\mathbb{C}[X_1,\ldots,X_d]_{(X_1,\ldots,X_d)}$. In this case, we are done by Proposition \ref{generalize for depth}.
\end{proof}

\section{The mixed characteristic case} \label{Mixed}
The purpose of this section is to explore the HLY cojecture for regular local rings of mixed characteristic. More precisely, in what follows, let $(R,\fm, \mathbb{K})$ be a local ring of mixed characteristic with $\charac \mathbb{K}=p > 0$. It is called ramified if $p \in \fm^2$
and it is unramified if $p$ is not in $\fm^2$.

In mixed characteristic, what is known about the HLY conjecture is the following theorem proved by Zhang \cite[Theorem 1.3]{ZhangFmodules}.

\begin{teo} [HLY in mixed characteristic] \label{zhangfailure}
Let $R=V[\![x_2,\ldots,x_d]\!]$ be a formal power series ring over a complete discrete valuation ring $(V,\pi V,\mathbb{K})$ of mixed characteristic $(0,p)$, let $I\subset R$ be an ideal, and let $j\geq 0$ be an integer. In addition, assume that $\mathbb{K}$ is a finite--dimension $\mathbb{K}^p$--vector space. If either $\ker(\xymatrix@1{H_I^j (R)\ar[r]^-{\cdot\pi}& H_I^j (R)})\neq 0$ or $\Coker(\xymatrix@1{H_I^j (R)\ar[r]^-{\cdot\pi}& H_I^j (R)})\neq 0$, then we have that $\spec (R/\pi R)\subseteq\stsupp_R (H_I^j (R)^{\vee})$.  
\end{teo}

Keeping in mind Zhang's statement, it is clear that the remaining cases to study the conjecture in mixed characteristic is when the uniformizer $\pi$ of the base ring is either regular or coregular on $H_I^j (R)$. Recall that coregular sequences have been studied systematically in the setting of local cohomology modules by Hartshorne and Polini in \cite{HartshornePolinicoregular}.

\begin{lm}\label{first step towords HLY in mixed characteristic}
Let $(R,\fm)$ be a Gorenstein local domain of dimension $d$, let $I\subset R$ be an ideal, and let $0\neq r\in R$ such that $\overline{R}:=R/rR$ is a Gorenstein local ring. Assume that there is an integer $j\geq 1$ such that:

\begin{enumerate}[(i)]

\item $H_I^j (R)\neq 0$.

\item Either $H_I^{j+1}(R)=0$ or the multiplication by $r$ $\xymatrix@1{H_I^{j+1}(R)\ar[r]^-{\cdot r}& H_I^{j+1}(R)}$ is injective.

\item $(0)\in\coass_{\overline{R}}(H_I^j (\overline{R}))$.

\end{enumerate}
Then, there is a surjective $R$--homomorphism $\xymatrix@1{H_I^j (R)\ar@{->>}[r]& \overline{E}}$, where
\[
\overline{E}:=\left\{\xi\in H_{\fm}^d (R):\ r\cdot\xi=0\right\}.
\]
\end{lm}

\begin{proof}
First, the short exact sequence $\xymatrix@1{0\ar[r]& R\ar[r]^-{\cdot r}& R\ar[r]& \overline{R}\ar[r]& 0}$ induces the following exact sequence of $I$--supported local cohomology modules.
\[
\xymatrix{H_I^j(R)\ar[r]^-{\cdot r}& H_I^j(R)\ar[r]& H_I^j (\overline{R})\ar[r]& H_I^{j+1}(R)\ar[r]^-{\cdot r}& H_I^{j+1}(R).}
\]
Our assumption (ii) implies that this exact sequences reduces to the following one.
\[
\xymatrix{H_I^j(R)\ar[r]^-{\cdot r}& H_I^j(R)\ar[r]& H_I^j (\overline{R})\ar[r]& 0.}
\]
On the other hand, our assumption (iii) combined with the fact that $\overline{R}$ is Gorenstein of dimension $d-1$ implies that there is a surjective $\overline{R}$--module map $\xymatrix@1{H_I^j (\overline{R})\ar@{->>}[r]& H_{\fm}^{d-1}(\overline{R})}$. Now, we also observe that we have the following exact sequence
\[
\xymatrix{H_{\fm}^{d-1}(R)\ar[r]& H_{\fm}^{d-1}(\overline{R})\ar[r]& H_{\fm}^d (R)\ar[r]^-{\cdot r}& H_{\fm}^d (R)\ar[r]& H_{\fm}^d (\overline{R}).}
\]
Since both $R$ and $\overline{R}$ are Gorenstein local of dimension $d$ and $d-1$ respectively, we have that $H_{\fm}^{d-1}(R)=0=H_{\fm}^d (\overline{R})$. In this way, this exact sequence boils down to the following short one.
\begin{equation}\label{second surjection}
\xymatrix{0\ar[r]& H_{\fm}^{d-1}(\overline{R})\ar[r]& H_{\fm}^d (R)\ar[r]^-{\cdot r}& H_{\fm}^d (R)\ar[r]& 0.}
\end{equation}
Thus, \eqref{second surjection} implies that
\[
H_{\fm}^{d-1}(\overline{R})=\left\{\xi\in H_{\fm}^d (R):\ r\cdot\xi=0\right\}.
\]
Summing up, by combining the surjection $\xymatrix@1{H_I^j (\overline{R})\ar@{->>}[r]& H_{\fm}^{d-1}(\overline{R})}$ with the exactness of \eqref{second surjection} we obtain the desired conclusion.
\end{proof}

\begin{lm}\label{second step}
Let $(R,\fm)$ be a Gorenstein local domain of dimension $d$, let $I\subset R$ be an ideal, and let $0\neq r\in R$ such that $\overline{R}:=R/rR$ is a Gorenstein local ring. Assume that there is an integer $j\geq 1$ such that:

\begin{enumerate}[(i)]

\item $H_I^j (R)\neq 0 \neq H_I^j (\overline{R})$.

\item  Either $H_I^{j+1}(R) = 0$ or the multiplication by $r$ $\xymatrix@1{H_I^{j+1}(R)\ar[r]^-{\cdot r}& H_I^{j+1}(R)}$ is injective.
\end{enumerate}
Then, $(0)\notin\coass_{R}(H_I^j (\overline{R}))$. In addition, if $\lvert\cosupp (H_I^j (R))\rvert=1$, then $(0)\notin\coass_{R}(H_I^j (R))$.
\end{lm}

\begin{proof}
First, the short exact sequence $\xymatrix@1{0\ar[r]& R\ar[r]^-{\cdot r}& R\ar[r]& \overline{R}\ar[r]& 0}$ induces the following exact sequence of $I$--supported local cohomology modules.
\[
\xymatrix{H_I^j(R)\ar[r]^-{\cdot r}& H_I^j(R)\ar[r]& H_I^j (\overline{R})\ar[r]& H_I^{j+1}(R)\ar[r]^-{\cdot r}& H_I^{j+1}(R).}
\]
Our assumption (ii) implies that this exact sequence reduces to the following one.
\begin{equation} \label{1stshort}
 \xymatrix{H_I^j(R)\ar[r]^-{\cdot r}& H_I^j(R)\ar[r]& H_I^j (\overline{R})\ar[r]& 0.}   
\end{equation}
On the other hand, our assumption (ii) combined with (\ref{1stshort}) implies the following commutative diagram:
\begin{equation}
\xymatrix{H_{I}^j (R)\ar[r]^-{\cdot r}\ar@{=}[d]& H_{I}^j (R) \ar@{=}[d] \ar[r]& H_I^j (\overline{R}) \ar[r]\ar[d]& 0\ar@{=}[d]\ar@{=}[r]& 0\ar@{=}[d]\\ H_{I}^j (R)\ar[r]^-{\cdot r}& H_{I}^j (R) \ar[r]& H_I^j (R)/rH_I^j (R) \ar[r]& 0\ar@{=}[r]& 0.}
\end{equation}
Using the Five Lemma, there is an isomorphism $ H_I^j (\overline{R}) \cong H_I^j (R)/rH_I^j (R)$ of $R$-modules. As both of $H_I^j (R)$ and $H_I^j (\overline{R})$ are non-zero, by virtue of \cite[Folgerung 3.2]{Zoschingerminimax} we have that
$$\emptyset \neq  \coass_{R}(H_I^j (\overline{R})) =\coass_{R}(H_I^j (R)) \cap V(r).$$
An immediate consequence of this is that $(0)\notin\coass_{R}(H_I^j (\overline{R}))$. This finishes the proof. 
\end{proof}

\subsection{The unramified case}

Although Zhang's result cover the unramified case, we present the following result which does not rely on Lyubeznik--Yildirim theorem on prime characteristic.

\begin{lm}\label{reductions in the unramified case}
Let $d\geq 2$ be an integer, let $\mathbb{K}$ be a perfect field of prime characteristic $p$, let $W(\mathbb{K})$ be the ring of Witt vectors attached to $\mathbb{K}$, set $R:=W(\mathbb{K})[\![x_2,\ldots,x_d]\!]$, and let $I\subset R$ be an ideal such that $p\in I$. Write $I=J+(p)$, where $p\notin J$ and $J\subset R$ is an ideal. Let $j\geq 1$ be an integer such that $H_I^j (R)\neq 0$. Then, the following assertions hold.

\begin{enumerate}[(i)]

\item If $H_{(p)}^0 (H_J^{j-1}(R))\neq 0$, then $\coass_R (H_{(p)}^0 (H_J^{j-1}(R)))\subseteq\coass_R (H_I^j (R))$.

\item If $H_{(p)}^0 (H_J^{j-1}(R))=0$, then $\coass_R (H_I^j (R))\subseteq\coass_R (H_{JR_p}^{j-1}(R_p))$.

\end{enumerate}
    
\end{lm}

\begin{proof}
By \cite[8.1.2 (ii)]{BroSha}, we have the short exact sequence
\[
\xymatrix{0\ar[r]& H_{(p)}^1 (H_J^{j-1}(R))\ar[r]& H_I^j (R)\ar[r]& H_{(p)}^0 (H_J^{j-1}(R))\ar[r]& 0.}
\]
In this way, part (i) follows immediately from this fact.

Assume now that $H_{(p)}^0 (H_J^{j-1}(R))=0$. By the above exact sequence, we have an $R$--module isomorphism
\[
H_I^j (R)\cong H_{(p)}^1 (H_J^{j-1}(R)\cong H_{(p)}^1 (R)\otimes_R H_J^{j-1}(R).
\]
Since $p$ is a non--zero divisor on $R$, we have the following short exact sequence.
\[
\xymatrix{0\ar[r]& R\ar[r]& R_p\ar[r]& H_{(p)}^1 (R)\ar[r]& 0.}
\]
Tensoring this exact sequence with $H_J^{j-1}(R)$, and keeping in mind that $R_p$ is a flat $R$--module, we obtain the following exact sequence.
\[
\xymatrix{0\ar[r]& \Tor_1^R (H_{(p)}^1 (R),H_J^{j-1}(R))\ar[r]& H_J^{j-1}(R)\ar[r]& R_p\otimes_R H_J^{j-1}(R)\ar[r]& H_I^j (R)\ar[r]& 0.}
\]
Combining the exactness of this sequence jointly with flat base change for local cohomology modules, we can conclude in this case that $\coass_R (H_I^j (R))\subseteq\coass_R (H_{JR_p}^{j-1}(R_p))$. The proof is therefore completed.
\end{proof}
Zhang in \cite[Theorem 1.4]{SVTZhang} obtained the unramified version of the second vanishing Theorem, which states that for a $d$-dimensional complete unramified regular local ring of mixed characteristic, whose residue field is separably closed, the vanishing of 
$H_I^{i} (R)$ for all $i\geq d-1$ is equivalent to connectedness of the punctured spectrum of $R/I$, where $\dim (R/I) \geq 2$. A partial second vanishing Theorem in the ramified case was obtained in \cite{kazumaSVT}.
In this way, we review the following statement obtained by Boix and Eghbali in \cite[Theorem 1.1(3)]{BoixEghbali2023IJAC}.

\begin{teo}
Let $(V,pV,\mathbb{K})$ be a complete unramified Noetherian DVR of mixed characteristic $p >0$, and let $R=V[\![x_1, \ldots, x_n]\!]$.
If $I\subset R$ is an ideal of pure dimension two, then $(0) \in \coass_R(H^{n-1}_I (R))$ if and only if multiplication by $p$ on $H^{n-1}_I (R)$ is surjective.

In other words, HLY holds at $n-1$ if and only if $p$ is a coregular element on $H^{n-1}_I (R)$.
\end{teo}

Next statement roughly says that, under some conditions, HLY does not hold in general in mixed characteristic. The precise statement reads as follows.

\begin{prop}
Let $(V,pV,K)$ be a complete unramified Noetherian DVR of mixed characteristic $p >0$ and let $R=V[\![x_1, \ldots, x_n]\!]$. Suppose that multiplication by $p$ on $H_I^{n-j}(R)$ is surjective and $\dim_K(\Ext^i_R(K,H_I^{n-j}(R)) \neq 0$ and finite for all $i \geq 0$. Assume that $[K : K^p]< \infty$. Then, the HLY conjecture does not hold for $H_I^{n-j}(R)$.
\end{prop} 

 \begin{proof}
Our assumption that $p$ is a coregular element on $H_I^{n-j}(R)$ is equivalent to the following short exact sequence 
\[
\xymatrix{0\ar[r]& (0:_{H_I^{n-j}(R)} (pR))\ar[r]& H_I^{n-j}(R)\ar[r]^-{\cdot p}& H_I^{n-j}(R)\ar[r]& 0.}
\]
Moreover, thanks to \cite[Lemma 4.1]{BetancourtWitt2012mixed}, we know that $\dim_K(\Ext^i_{R/pR}(K,(0:_{H_I^{n-j}(R)} (pR))) \neq 0$, what is equivalent to say that $p$ is a zerodivisor on $H_I^{n-j}(R)$ so $\ker(\xymatrix@1{H_I^{n-j}(R)\ar[r]^-{\cdot p}& H_I^{n-j}(R)}) \neq 0$. Now, we are done by Theorem \ref{zhangfailure}.     
 \end{proof}
We end this part by briefly exploring the HLY conjecture for cohomologically full rings, which were introduced and systematically studied by Dao, De Stefani and Ma in \cite{DaoDeStefaniMacohfull}.

\begin{disc}[Connection with cohomologically full rings]
Let $d\geq 2$ be an integer, let $\mathbb{K}$ be a perfect field of prime characteristic $p$, set $R:=W(\mathbb{K})[\![x_2,\ldots,x_d]\!]$, and let $I\subset R$ be an ideal such that $p\notin I$. Fix an integer $j\geq 0$, and assume that $R/I$ is $j$--th cohomologically full \cite[Definition 1.1]{DaoDeStefaniMacohfull}. This is equivalent to say, by \cite[Proposition 2.1 (5)]{DaoDeStefaniMacohfull}, that we have, setting $K^j (R/I):=\Ext_R^{d-j} (R/I,R)$, a short exact sequence of $R$--modules
\[
\xymatrix{0\ar[r]& K^j (R/I)\ar[r]& H_I^{d-j}(R)\ar[r]& Q\ar[r]& 0,}
\]
where $Q$ is the cokernel of the injection $\xymatrix@1{K^j (R/I)\ar@{^(->}[r]& H_I^{d-j}(R)}$. By \cite[Theorem 2.7 (vi)]{Richardsoncosupport}, we know that
\[
\cosupp_R (K^j (R/I))=\begin{cases}
\{\mathfrak{m}\},\text{ if }K^j (R/I)\neq 0,\\
\emptyset,\text{ if }K^j (R/I)=0.
\end{cases}
\]
This shows, by applying \cite[Theorem 2.7 (v)]{Richardsoncosupport}, that
\[
\cosupp_R (Q)\subseteq\cosupp_R (H_I^{d-j}(R))\subseteq\{\fm\}\cup\cosupp_R (Q).
\]
Summing up, we have shown that $(0)\in\coass_R (H_I^{d-j}(R))$ if and only if $(0)\in\coass_R (Q)$.
\end{disc}

\subsection{The ramified case}

The goal of this part, as the title says, is to focus on the case where our mixed characteristic base ring is ramified. Let us recall some known facts about Eisenstein extensions of normal local rings from \cite[Sec. 29, pp. 229]{Mat86}. Let $(R, \fm)$ be a normal local ring  and let $B= R[X]/f(X)$ be the extension ring defined by the Eisenstein polynomial $f(X) = X^{n}+a_1X^{n-1}+ \ldots +a_n$ with $a_i \in \fm$ for every $i = 1, \ldots , n$ and $a_n\notin\fm^2$. In this setting, the local ring $(B, \fn)$ is called an Eisenstein extension of $R$.

We have borrowed the below result from \cite[page 2]{Bhataramified}. Anyway, we add some extra information, as the reader can appreciate in the following statement.

\begin{lm} \label{Eisenstein}
The following statements hold.

\begin{enumerate}[(i)]

\item An Eisenstein extension of a regular local ring is regular local.

\item An Eisenstein extension of an unramified regular local ring is a ramified regular local ring.

\item Any complete regular local ring is an Eisenstein extension of some unramified complete regular local ring.

\item The Eisenstein extensions mentioned in (i) and (ii) are faithfully flat.  
\end{enumerate}
\end{lm}

\begin{proof}
Part (i) is just \cite[Theorem 29.8 (i)]{Mat86}. Part (ii) is \cite[Theorem 29.8 (ii)]{Mat86}. Part (iii) is precisely \cite[Theorem 17]{Cohen1946}. Finally, part (iv) is \cite[Theorem 23.1]{Mat86}.
\end{proof}

Eisenstein extensions were considered by Bhattacharyya in \cite[Corollary 1]{Bhata2024} to obtain some partial answers to the HLY for ramified complete regular local rings. The interested reader is encouraged to consult \cite[Corollary 1]{Bhata2024} to see the precise statement.

We conclude this section with the following observation, which roughly says that, under certain assumptions, HLY conjecture for a ramified regular local ring can be studied by reduction to the unramified setting.

\begin{rk}
Let $(R, \fm)$ be an unramified regular local ring and $(S,\fn)$ be an Eisenstein extension of $R$ such that $\sqrt{\fm S}=\fn$. Moreover, let $I\subset S$ be an ideal and let $j\geq 0$ be an integer. Then, as a by-product of Lemmas \ref{ascending HLY} and \ref{Eisenstein}, we may reduce the ramified setting to the unramified one provided we restrict the attention to ideals that are extended from $R$.   
\end{rk}

\section{On the failure of HLY in characteristic zero}\label{section: HLY counterexamples}

The main goal of this section is to show that HLY is in general not true for regular local rings containing a field of characteristic zero.

Let $\mathbb{K}$ be a field of characteristic zero and let $R=\mathbb{K}[x_1, \ldots, x_d]$ be the polynomial ring over $\mathbb{K}$ with its standard grading and $\fm=(x_1, \ldots, x_d)$ be the maximal ideal. Suppose that $E=H^d_{\fm}(R)$, the graded top local cohomology module of $R$ supported in $\fm$, is an injective hull of $\mathbb{K}=R/\fm$. For a graded $R$-module $M$, the graded Matlis dual of $M$ denoted by $M^{\vee}=\hbox{}^{\ast}\Hom_R(M,E)$ is the graded $R$-module. Once again, $\mathcal{D}$ will denote the ring of $\mathbb{K}$--linear differential operators on $R$.

\begin{df}
Let $M$ be a left (right) $\mathcal{D}$-module whose underlying $R$-module is $M=\oplus_{l \in \mathbb{Z}}M_l$. $M$ is called a graded $\mathcal{D}$-module if for all $l \in \mathbb{Z}$ and $1 \leq i \leq  n$ one has $\partial_i(M_l) \subseteq M_{l-1}$.
\end{df}

We also want to recall here the so--called holonomic duality.

\begin{df}\label{holonomic dual: definition}
Let $M$ be a left holonomic $\mathcal{D}$-module. The module $\Ext^n_{\mathcal{D}}(M,\mathcal{D})$ is a right $\mathcal{D}$-module. Set $M^*$ as the left $\mathcal{D}$--module attached to $\Ext^n_{\mathcal{D}}(M,\mathcal{D})$. It is known \cite[2.6.8 (iv)]{HoTaTa} to be holonomic, and called the \textit{holonomic dual} of $M$. It is known that holonomic duality is a contravariant functor from the category of holonomic left (or right) $\mathcal{D}$-modules to the category of holonomic right (or left) $\mathcal{D}$-modules, and that $M\cong M^{\ast\ast}$ \cite[2.6.5]{HoTaTa}.
\end{df}
Although we think next lemma is well known to experts, we provide an elementary proof because of the lack of a reference.

\begin{lm}\label{holonomic duality preserves support}
Let $M$ be a holonomic $\mathcal{D}$-module. Then, we have that $\stsupp M =\stsupp M^{\ast}.$ 
\end{lm}

\begin{proof}
Note that, by the own definition of the holonomic dual, $\stsupp M^{\ast} \subseteq \stsupp M$. On the other hand, since $M\cong M^{\ast\ast}$, we have in particular that $\stsupp M =\stsupp M^{\ast \ast}$. In this way, the claim is proved by $\stsupp M^{\ast \ast} \subseteq \stsupp M^{\ast} \subseteq \stsupp M$.
\end{proof}

\begin{rk}\label{support and characteristic variety}
Let $M$ be a holonomic left $\mathcal{D}$--module. It is known \cite[2.6.12]{HoTaTa} that both $M$ and $M^{\ast}$ have the same characteristic variety. However, this is not enough to conclude that they have the same support, as it is pointed out in the following Mathoverflow discussion:
\begin{verbatim}
https://tinyurl.tools/50049f2e
\end{verbatim}
In other words, we can not deduce Lemma \ref{holonomic duality preserves support} from that fact.
\end{rk}

The following result proved by Hartshorne and Polini in \cite[Theorem 6.2]{HartshornePolinisimple} is essential in our attempt to show the failure cases of the HLY.

\begin{prop}\label{transition from graded to formal deRham stuff}
Let $\mathbb{K}$ be a field of characteristic zero, let $R=\mathbb{K}[x_0,\ldots,x_n]$ endowed with its standard $\mathbb{N}_0$--grading, let $I\subset R$ be a homogeneous ideal, and let $j\geq 0$ be an integer. Then, for any integer $i\geq 0$, the natural completion map $\xymatrix@1{H_{dR}^i (H_I^j (R))\ar[r]& H_{dR}^i (H_{I\widehat{R}}^j (\widehat{R}))}$
is an isomorphism.
\end{prop}

Another ingredient we plan to use here is the following nice connection between de Rham cohomology and Lyubeznik numbers for regular holonomic $\mathcal{D}$--modules proved by Reichelt, Walther and Zhang \cite[Lemma 1.4]{ReicheltWaltherZhang}. We only state this result in the generality we need. For the definition of regular holonomic $\mathcal{D}$--module the reader is referred to \cite[Chapter 6]{HoTaTa}.

\begin{lm}\label{de Rham and Lyubeznik numbers}
Let $R=\mathbb{C}[x_1,\ldots,x_d]$, let $\fm=(x_1,\ldots,x_d)$, and let $M$ be a graded $R$--module which is also a regular holonomic $\mathcal{D}$--module. Then, for any integer $i\geq 0$ we have that
\[
\dim_{\mathbb{C}}(H_{dR}^i (M))=\dim_{\mathbb{C}} (\Hom_R (\mathbb{C},H_{\fm}^{d-i} (M^{\ast})).
\]
\end{lm}

The key result in order to prove the potential failure of the HLY conjecture in characteristic zero is given in the following statement. This is the main result of this section

\begin{teo} \label{1stfailure}
Let $R=\mathbb{C}[x_1, \ldots, x_d]$, and $I$ be a homogeneous non-$\fm$-primary ideal of $R$. Assume that $H^j_I(R) \neq 0$ and, in addition, suppose that $H^0_{\fm}(H^j_I(R)^*)=0$. Then, we have that HLY conjecture does not hold for $H^j_{I\hat{R}}(\hat{R})$.
\end{teo}

\begin{proof}
Keeping in mind that $H_I^j (R)$ is a regular holonomic $\mathcal{D}$--module and $I$ is a homogeneous ideal, we can use Lemma \ref{de Rham and Lyubeznik numbers} to conclude that
\[
\dim_{\mathbb{C}} (H_{dR}^d (H_I^j (R)))=\dim_{\mathbb{C}} (\Hom_R (\mathbb{C},H^0_{\fm}(H^j_I(R)^{\ast}))=0.
\]
Summing up, we have shown that $H^{d}_{dR}(H^j_I(R))=0$. Now, we are done by combining Corollary \ref{derhamdim} and Proposition \ref{transition from graded to formal deRham stuff}. The proof is therefore completed.
\end{proof}
We conclude this section by raising a question that is intimately connected with Theorem \ref{1stfailure}, as the reader can easily appreciate.

\begin{quo}
Under which conditions we can guarantee that the holonomic dual of a $D$--module is $\mathfrak{m}$--torsion free?
\end{quo}

\section{HLY conjecture and pure morphisms}\label{section: HLY and squarefree monomial ideals}
The main purpose of this section is to prove the following:

\begin{teo}[HLY holds for squarefree monomial ideals]\label{HLY for squarefree monomial ideals}
Let $\mathbb{K}$ be any field, let $R=\mathbb{K}[\![x_1,\ldots,x_d]\!]$, and let $I\subset R$ be a squarefree monomial ideal. Then, for any integer $j\geq 0$ with $H_I^j (R)\neq 0$, we have that $(0)\in\coass_R (H_I^j (R))$.
\end{teo}
The proof of this result will be an adaptation of the proof presented by Lyubeznik and Yildirim in the prime characteristic setting, using for that purpose some of the results obtained by Singh and Walther in \cite{SinghWalther2007}.

\begin{nt}\label{setup for flat maps}
Along this section, $R=\mathbb{K}[\![x_1,\ldots,x_d]\!]$,  $\xymatrix@1{R\ar[r]^{\varphi}& R}$ will denote a flat local ring endomorphism, and let $\Phi$ be the functor as defined in \cite[Remark 2.6]{SinghWalther2007}. That is, for any $R$--module $M$, $\Phi (M)=\varphi_* (M)\otimes_R M$, where $\varphi_* (M)$ denotes the underlying abelian group $M$ regarded as left $R$--module via restriction of scalars under $\varphi$. In this setting, we have a natural $R$--module map $\xymatrix@1{M\ar[r]^-{\psi_M}& \Phi (M)}$ sending $m\in M$ to $\varphi_* 1\otimes m$.
\end{nt}

Our next definition is a mild generalization of Lyubeznik $\mathcal{F}$--modules \cite{Lyu1997} in this setting.

\begin{df}
We say that $\mathcal{M}$ is a \textit{$\Phi$--module} if it is an $R$--module that admits an $R$--module isomorphism $\xymatrix@1{\mathcal{M}\ar[r]^-{\sim}& \Phi(\mathcal{M})}$. On the other hand, we say that $\mathcal{M}$ is a \textit{finite $\Phi$--module} if there is a finitely generated $R$--module $M$ and an injective $R$--module homomorphism $\xymatrix@1{M\ar[r]^-{\beta}& \Phi (M)}$ such that
\[
\mathcal{M}\cong\colim\left(\xymatrix{M\ar[r]^-{\beta}& \Phi (M)\ar[r]^{\Phi (\beta)}& \Phi^2 (M)\ar[r]& \ldots}\right)
\]
Here, $\colim$ means direct limit. In that situation, we say that $M$ is a \textit{root} of $\mathcal{M}$.
\end{df}
Before going on, we review the following technical statement.

\begin{lm}\label{equivalence on Matlis reflexive}
The natural transformation $\tau: \Phi\circ (-)^{\vee}\Longrightarrow (-)^{\vee}\circ\Phi$ is an equivalence when restricted to the class of Matlis reflexive $R$--modules. 
\end{lm}

\begin{proof}
By Enochs--Zink Theorem, a module $M$ is Matlis reflexive if and only if it fits into a short exact sequence $\xymatrix@1{0\ar[r]& M'\ar[r]& M\ar[r]& M''\ar[r]& 0,}$ where $M'$ is a finitely generated $R$--module and $M''$ is an Artinian $R$--module. Since both $\Phi\circ (-)^{\vee}$ and $(-)^{\vee}\circ\Phi$ are exact, we end up with the following commutative diagram with exact rows.
\[
\xymatrix{0\ar[r]& \Phi (M'^{\vee})\ar[d]_{\tau_{M'}}\ar[r]& \Phi (M^{\vee})\ar[d]_{\tau_M}\ar[r]& \Phi (M''^{\vee})\ar[d]_{\tau_{M''}}\ar[r]& 0\\
0\ar[r]& \Phi (M')^{\vee}\ar[r]& \Phi (M)^{\vee}\ar[r]&\Phi (M'')^{\vee}\ar[r]& 0.}
\]
In this way, by Five's Lemma, it is enough to show that $\tau$ is an equivalence when restricted to the class of finitely generated and Artinian $R$--modules.

On the one hand, let $M'$ be a finitely generated $R$--module, and fix a presentation
\[
\xymatrix{R^{\oplus b_1}\ar[r]& R^{\oplus b_0}\ar[r]& M\ar[r]& 0.}
\]
Once again by Five's Lemma jointly with exactness, in order to see that $\tau$ is an equivalence for finitely generated $R$--modules, it is enough to show that $\Phi (R^{\vee})\cong \Phi (R)^{\vee}$. And this is true, as shown by Singh and Walther  in \cite[Remark 2.6 (3)]{SinghWalther2007}.

On the other hand, a similar argument applies for the class of Artinian $R$--modules, mainly because in a local ring, a module $M''$ is Artinian if and only if it fits into a short exact sequence
\[
\xymatrix{0\ar[r]& M''\ar[r]& E^{\oplus m_0}\ar[r]& E^{\oplus m_1}}.
\]
The proof is therefore completed.
\end{proof}
Actually, we plan to prove the following mild generalization of \cite[Theorem 1.1]{LyubeznikYildirim}. As the reader will easily see, we only plan to single out along the proof the main differences between our setting and the proof of \cite[Theorem 1.1]{LyubeznikYildirim}, because Lyubeznik--Yildirim's argument almost works in our situation up to the changes we plan to explain now.

\begin{teo}\label{LY for phi-modules}
If $\mathcal{M}$ is a finite $\Phi$--module and $(0)\notin\Ass_R (\mathcal{M})$, then $(0)\in\coass_R (\mathcal{M})$.
\end{teo}

\begin{proof}
By Lemma \ref{equivalence on Matlis reflexive}, we can guarantee that there is a natural equivalence of functors $\Phi\circ (-)^{\vee}\cong (-)^{\vee}\circ\Phi$ when restricted to the class of Matlis--reflexive $R$--modules. This implies that \cite[Lemma 3.2]{LyubeznikYildirim} holds in our setting, meaning that we can write
\[
\mathcal{M}^{\vee}=\lim\left(\xymatrix{N& \Phi (N)\ar[l]_{\alpha}& \Phi^2 (N)\ar[l]_{\Phi (\alpha)}& \ldots\ar[l]}\right),
\]
where $N$ is an Artinian $R$--module, $\alpha$ is surjective and $\ker (\alpha)\neq 0$.

Now, we want to prove \cite[Lemma 3.2]{LyubeznikYildirim}. In this case, we need to assume that $\varphi$ is the map of $\mathbb{K}$--algebras such that $\varphi (x_i)=x_i^t$ for some $t\geq 1$ and any $1\leq i\leq d$. Under this additional assumption, we can define an element $b_{\bullet}=(b_k)_{k\geq 1}\in\mathcal{M}^{\vee}$ as follows. For $k=1$, choose $b_1$ a non--zero element in the socle of $\ker (\alpha)$, which exists because $\ker (\alpha)$ is a non--zero Artinian module. For $k\geq 2$, set $b_k:=\psi_{\Phi^{k-1}(M)}(b_{k-1})$, where $\psi$ is as in Notation \ref{setup for flat maps}. With this choice, we can ensure that $(0:_R b_k)=\fm^{[t^{k-1}]}$ for any $k\geq 1$. Let us prove this last claim for the convenience of the reader. For $k=1$, it is clear that $(0:_R b_1)=\fm$ because $b_1$ belongs to the socle of $\ker (\alpha)$. For $k\geq 2$, we only have to keep in mind that $(0:_R b_k)=(0:_R b_{k-1})^{[t]}$. Our claim follows combining that two facts jointly with increasing induction on $k$.

The next step is \cite[Lemma 3.3]{LyubeznikYildirim}, which in our setting reads as follows: given $y\in\fm-\fm^k$, we have that $(0:_R yb_k)\subset\fm^{t^{k-1}-k}$. In particular, $(0:_R yb_k)\subseteq\fm^k$ for $k\geq 4$. The proof of \cite[Lemma 3.3]{LyubeznikYildirim} also works in our setting, because our ambient ring is a formal power series ring over a field, which in particular satisfies the two properties we need: an integral domain where we can perform division. In this way, all the arguments given in \cite[Lemma 3.3]{LyubeznikYildirim} hold, just keeping in mind that we have to write in our case $t^{k-1}$ instead of $p^{k-1}$.

Finally, we come to the proof, which is also a mild adaptation of \cite[Theorem 1]{LyubeznikYildirim}. First, it is possible to construct $n'_{\bullet}\in\mathcal{M}^{\vee}$ such that $(0:_R n'_k)\subseteq\fm^k$ for any $k\geq 4$ with the same argument given along the proof of \cite[Theorem 1]{LyubeznikYildirim}. The fact that this element has a zero annihilator comes at the end to the use of Krull's intersection theorem \cite[Corollary 5.4]{Eisenbud1995}, which also works in our setting because we are working on a Noetherian local ring. The proof is therefore completed.
\end{proof}

Now, we finally come to the proof of the main result of this section.

\begin{proof}[Proof of Theorem \ref{HLY for squarefree monomial ideals}]
Set $\varphi$ as the $\mathbb{K}$--algebra endomorphism given by raising each variable to its square. In this case, we have that $H_I^j (R)$ is a finite $\Phi$--module with $\Ext_R^j (R/I,R)$ as root. In this way, the result follows immediately from Theorem \ref{LY for phi-modules}.
\end{proof}

\section*{Acknowledgements}
Majid Eghbali received partial support by a grant from IPM (No.1404130017). Alberto F. Boix received partial support by a grant PID2022-137283NB-C22 funded by the Spanish Ministerio de Ciencia, Innovacci\'on y Universidades (MICIU/AEI/10.13039/501100011033). The second named author would like to dedicate the paper to Rostam--e--Dastaan.

\section*{Data sharing}
Data sharing not applicable because no new data were created or analyzed in this study.

\section*{Conflict of interest}
There is no conflict of interest.

\bibliographystyle{alpha}
\bibliography{AFBoixReferences}

\end{document}